\documentclass[12pt]{amsart}
\usepackage{amssymb, latexsym}

\theoremstyle{plain}
\newtheorem{theorem}{Theorem}[section]
\newtheorem{corollary}[theorem]{Corollary}
\newtheorem{proposition}[theorem]{Proposition}
\newtheorem{lemma}[theorem]{Lemma}

\theoremstyle{definition}
\newtheorem{definition}[theorem]{Definition}
\newtheorem{remark}[theorem]{Remark}
\newtheorem{notation}[theorem]{Notation}
\newtheorem{example}[theorem]{Example}
\newtheorem{problem}[theorem]{Problem}

\newtheorem{question}[theorem]{Question}

\numberwithin{equation}{section}

\usepackage{dsfont}
\usepackage{braket}
\usepackage{hyperref}

\begin{document}

\title[Partitions of unity and barycentric algebras]{Partitions of unity and barycentric algebras}

\author[A. Zamojska-Dzienio]{Anna Zamojska-Dzienio}
\address{Faculty of Mathematics and Information Science\\
Warsaw University of Technology\\
00-661 Warsaw, Poland}

\email{anna.zamojska@pw.edu.pl}


\begin{abstract}
Barycentric coordinates provide solutions to the problem of expressing an element of a compact convex set as a convex combination of a finite number of extreme points of the set. They have been studied widely within the geometric literature, typically in response to the demands of interpolation, numerical analysis and computer graphics. In this note we bring an algebraic perspective to the problem, based on \emph{barycentric algebras}. We focus on the discussion of relations between different subclasses of partitions of unity, one arising in the context of barycentric coordinates, based on the \emph{tautological map} introduced by Guessab.
\end{abstract}

\maketitle

\section{Introduction}
Let $k$, $n$ be positive integers such that $k< n$. Consider a convex polytope $\Pi$ in $\mathbb{R}^k$ with vertex set $\mathbf v_1,\dots,\mathbf v_n$. Then, each element $\mathbf v$ of $\Pi$ may be expressed as a convex combination
\begin{align}\label{E:barcoord}
\mathbf{v}&=\lambda_1\mathbf v_1+\ldots+\lambda_n\mathbf v_n
\\ \label{E:convcomb1}
\mbox{with}\quad
1&=\lambda_1+\ldots+\lambda_n
\end{align}
and $\lambda_i \in [0,1]\subseteq\mathbb{R}$. If $\mathbf v$ and $\mathbf v_i$ are given by Cartesian coordinates in the vector space $\mathbb{R}^k$, the coefficients $\lambda_i$, called \emph{barycentric coordinates}, may be calculated by solving the system \eqref{E:barcoord},\eqref{E:convcomb1} of $k+1$ linear equations in the $n$ non-negative unknowns $\lambda_1,\dots,\lambda_n$. If $\Pi$ is not a simplex, then $k+1<n$, so the system is underdetermined, and the barycentric coordinates in \eqref{E:barcoord} are not uniquely specified.

The following problem appears in many applications of convex polytopes, for example in geometric modeling and computer graphics (see e.g. ~\cite{FloaterGBCA, FHK} and their references).

\begin{problem}\label{P:problem}
Given the set $V$ of vertices $\mathbf v_i$ of a convex polytope $\Pi$, find a uniform system to produce uniquely determined barycentric coordinates for any point $\mathbf v$ of $\Pi$.
\end{problem}

In \cite{RSZ1, RSZ2} the authors brought an algebraic perspective to the problem. They developed a general framework for the study of barycentric
coordinate systems on a given convex polytope, founded on the theory of \emph{barycentric algebras} introduced in the nineteen-fifties independently by M.H. Stone \cite{Stone} and H. Kneser \cite{Kneser} for the axiomatization of real convex sets.

The set of all barycentric coordinate systems on a convex polytope itself forms a cancellative barycentric algebra (a \emph{convex set}), which then becomes an object of study in its own right | compare \cite[Section 6]{RSZ1}. A barycentric coordinate system is usually characterized by two properties: the ``partition of unity'' property (understood here as a partition of the constant function $\mathds 1_\Pi$ where the output value is $1$) and the ``linear precision'' property (a partition of the identity function $1_\Pi$ on $\Pi$). Use of barycentric algebra shows that, in our context, the partition of unity property is actually a consequence of the linear precision property that does not need to be specified separately \cite[Lemma~2.23, Remark~6.4]{RSZ1}.

This article serves two purposes. The first is to give an introduction for readers who are not familiar with the theory of barycentric algebras (Sections~\ref{S:rba}--\ref{S:pcs}). The second is to investigate links between different subclasses of partitions of unity (Section \ref{S:tm}), giving an algebraic interpretation to results obtained by Guessab \cite{G}. In particular, we obtain an alternative proof that the set of all barycentric coordinate systems on a convex polytope forms a convex set (Corollary~\ref{C:KPi}).

The paper is self-contained, but for further reading in the background of universal algebra we recommend \cite{B, BS}, and for barycentric algebra theory \cite{RS85,Sm11,Modes}.

\section{Real barycentric algebras}\label{S:rba}
\subsection{The basic definition}
Denote by $I$ the closed real unit interval $[0,1]$ and by $I^{\circ}=I-\{0,1\}=]0,1[$ the open real unit interval.
For $p,r\in I^{\circ}$ define:
\begin{enumerate}
\item[$(\mathrm a)$] \emph{complementation}: $\overline{r}=1-r$;
\item[$(\mathrm b)$] \emph{dual multiplication}: $p\circ r=p+r-p\cdot r= \overline{(\overline p\cdot \overline r)}$.
\end{enumerate}
Note that the values $\overline{r}$ and $p\circ r$ are again in $I^{\circ}$.

\begin{definition}\cite{RS85,Sm11,Modes}
A \emph{(real) barycentric algebra} $\mathcal{A}=(A,\{\underline{p}\mid p\in I^{\circ}\})$ is defined as a 
set $A$ that is equipped with binary \emph{operations}
\begin{equation}\label{E:OpOfBaAl}
\underline{p}\colon A\times A\to A;(a,b)\mapsto \underline{p}(a,b)
\end{equation}
for each element $p\in I^{\circ}$. The operations \eqref{E:OpOfBaAl} are required to satisfy the properties of \emph{idempotence}
\begin{equation}\label{E:idemptnc}
\underline{p}(a,a)=a
\end{equation}
for $a$ in $A$, \emph{skew-commutativity}
\begin{equation}\label{E:skewcomm}
\underline{p}(a,b)=\underline{\overline{p}}(b,a)
\end{equation}
for $a$, $b$ in $A$, and \emph{skew-associativity}
\begin{equation}\label{E:skewasoc}
\underline{p}(\underline{r}(a,b),c)=\underline{r\circ p}(a,\underline{p/(r\circ p)}(b,c))
\end{equation}
for $a$, $b$, $c$ in $A$.
\end{definition}

\begin{notation}
In what follows, we denote the set $\{\underline{p}\mid p\in I^{\circ}\}$ of operations briefly by $\underline{I}^{\circ}$, and a real barycentric algebra by $\mathcal{A}=(A,\underline{I}^{\circ})$.
\end{notation}

Summarizing, a real barycentric algebras $\mathcal{A}=(A,\underline{I}^{\circ})$ is an algebraic structure (briefly: algebra) with continuum many binary operations that are indexed by $I^\circ$. All these operations satisfy conditions \eqref{E:idemptnc}-\eqref{E:skewasoc} for all elements of $A$. Since we are interested here only in the case $I^\circ\subset\mathbb{R}$, we will skip \emph{real} in the name and simply write \emph{barycentric algebra} for $\mathcal{A}$.

\subsection{Homomorphisms, subalgebras, products}
By Birkhoff's HSP Theorem, the class $\mathbf{B}$ of barycentric algebras, as an equational class, forms a \emph{variety} of algebras, i.e., a class closed under the taking of homomorphic images, subalgebras and direct products.
\begin{definition}\label{D:BaHomSet}
Suppose that $\mathcal{A}=(A,\underline{I}^{\circ})$ and $\mathcal{A}'=(A',\underline{I}^{\circ})$ are barycentric algebras.\footnote{Note that operations in $\mathcal{A}$ and $\mathcal{A}'$ needn't be the same. Sometimes we add a superscript, writing $\underline{p}^{\mathcal{A}}$ to make it clear, but usually we leave it off whenever we can get away with it.}
A function $f\colon A\to A';x\mapsto f(x)$ is said to be a \emph{barycentric} (\emph{algebra}) \emph{homomorphism} if
\begin{equation}\label{E:BaryAHom}
f(\underline p^{\mathcal{A}}(a,b))=\underline p^{\mathcal{A}'}(f(a),f(b))
\end{equation}
for all $a,b\in A$ and $p\in I^\circ$.
A homomorphism is an \emph{isomorphism} if it is bijective. We say that two barycentric algebras are isomorphic when there exists an isomorphism between them.
\end{definition}
Informally stated, a barycentric homomorphism \emph{preserves the operations} of the barycentric algebras. It is well-known that it preserves not only basic binary operations $\underline{p}$ in $\underline{I}^{\circ}$ but also all \emph{derived} (or \emph{term}) operations obtained as compositions of the basic ones (see Section \ref{Ss:Coco} below). The composite of two composable barycentric homomorphism is again a barycentric homomorphism.

\begin{definition}\label{D:Sb}
Suppose that $\mathcal{A}=(A,\underline{I}^{\circ})$ is a barycentric algebra, and $B$ is a subset of $A$.
If
$$
\forall\ p\in I^\circ\,,\
a\in B
\mbox{ and }
b\in B
\
\Rightarrow
\
\underline p(a,b)\in B\,,
$$
then $\mathcal{B}=(B,\underline{I}^{\circ})$ is a \emph{subalgebra} of $\mathcal{A}$. We denote the subalgebra relationship by $\mathcal{B}\leq \mathcal{A}$.
\end{definition}

The intersection of any collection of subalgebras of a barycentric algebra $\mathcal{A}$ is again a subalgebra of $\mathcal{A}$.

\begin{example}
\begin{enumerate}
\item[$(\mathrm a)$] Due to idempotence \eqref{E:idemptnc}, any singleton in $\mathcal{A}\in\mathbf{B}$ is a subalgebra of $\mathcal{A}$.
\item[$(\mathrm b)$] For a barycentric homomorphism $f\colon \mathcal{A}\to\mathcal{A}'$, the image of any subalgebra of $\mathcal{A}$ with respect to $f$ is a subalgebra in $\mathcal{A}'$, and similarly the preimage of any subalgebra of $\mathcal{A}'$ with respect to $f$ is a subalgebra in $\mathcal{A}$.
\end{enumerate}
\end{example}

\begin{definition}\label{D:Sbgen}
Let $\mathcal{A}$ be a barycentric algebra, with a subset $S\subseteq A$. Then the \emph{subalgebra} $\braket S$ \emph{generated} by $S$ is the intersection of all the subalgebras of $\mathcal{A}$ that contain $S$.
\end{definition}

\begin{proposition}\cite[Proposition 123]{RS85}\label{P:hs}
Let $\mathcal{A}$, $\mathcal{A}'$ be barycentric algebras, and let $S\subseteq A$. Then a barycentric homomorphism $f\colon \braket S\to \mathcal{A}'$ is specified uniquely by its restriction $f\colon S\to A'$ to $S$.
\end{proposition}

\begin{definition}\label{D:DiPr}
Suppose $\mathcal{A}=(A,\underline{I}^{\circ})$, $\mathcal{A}'=(A',\underline{I}^{\circ})$ are barycentric algebras. Then the \emph{(direct) product} $\mathcal{A}\times\mathcal{A}'$ is the algebra defined on the set $A\times A'$ with componentwise structure
\begin{equation*}
\underline p^{\mathcal{A}\times\mathcal{A}'}((a,a'),(b,b'))=\left(\underline p^\mathcal{A}(a,b)\,,\underline p^{\mathcal{A}'}(a',b')\right)
\end{equation*}
for all $p\in I^\circ$.
\end{definition}

Iterated products and powers are defined in the obvious way, since the construction of $\mathcal{A}\times\mathcal{A}'$ extends naturally to a direct product of a finite number of factors.

\subsection{Convex sets}
The variety $\mathbf{B}$ of barycentric algebras is generated by its \emph{cancellative} members  barycentric algebra which satisfy the additional property of \emph{cancellativity}:
$$
\underline{p}(a,b)=\underline{p}(a,c) \mathrel{\Rightarrow} b=c
$$
for all operations $\underline{p}$ of $\underline{I}^{\circ}$ and all elements $a$, $b$, $c$ in $A$. Cancellative barycentric algebras are called \emph{convex sets}.

The class $\mathbf{C}$ of (algebras isomorphic to) convex sets considered as algebras $(C,\underline{I}^{\circ})$ is closed under the formation of subalgebras and direct products, but is not closed under homomorphic images. Thus, it is not a variety.

\begin{remark}\label{D:BarycAlg}
A barycentric algebra can be defined as a homomorphic image of a convex set.
\end{remark}

\subsection{Barycentric algebras of functions}

Suppose that $X$ is a set, and that $\mathcal{A}'=(A',\underline{I}^{\circ})$ is a barycentric algebra.
We write $\mathbf{Set}(X,\mathcal{A}')$ for the set of all functions from $X$ to $A'$.
For $p\in I^\circ$, define the \emph{pointwise} or \emph{componentwise} operation
\begin{equation}\label{E:pontwise}
\underline p(f,g)\colon X\to A';x\mapsto \underline p(f(x),g(x))
\end{equation}
on elements $f,g$ of $\mathbf{Set}(X,\mathcal{A}')$.
The set $\mathbf{Set}(X,\mathcal{A}')$ becomes a barycentric algebra under the pointwise operations \eqref{E:pontwise}. 

For $\mathcal{B}\leq\mathcal{A}'$ one obtains $\mathbf{Set}(X,\mathcal{B})\leq\mathbf{Set}(X,\mathcal{A}')$. Moreover, if $\mathcal{A}'\in\mathbf{C}$ then $\mathbf{Set}(X,\mathcal{A}')\in\mathbf{C}$.

Suppose now that $\mathcal{A}=(A,\underline{I}^{\circ})$ and $\mathcal{A}'=(A',\underline{I}^{\circ})$ are barycentric algebras. Write $\mathbf B(\mathcal{A},\mathcal{A}')$ for the set of all barycentric homomorphisms from $\mathcal{A}$ to $\mathcal{A}'$. The set $\mathbf B(\mathcal{A},\mathcal{A}')$ becomes a barycentric algebra under the pointwise operations \eqref{E:pontwise}, a subalgebra of the barycentric algebra $\mathbf{Set}(A,\mathcal{A}')$.

\subsection{Basic examples}

For our purposes the basic examples of a bary\-centric algebra, i.e. constructed from real vector spaces, are crucial.

\begin{example}\label{Ex:VectSpace}
Let $\mathcal{V}$ be a vector space over $\mathbb{R}$. Let $\underline{p}\colon \mathcal{V}\times \mathcal{V}\rightarrow \mathcal{V}$ for $p\in I^{\circ}$ be the \emph{weighted mean operation}:
\begin{equation}
\underline{p}(u,v)=(1-p)\cdot u+p\cdot v.
\end{equation}
We obtain the (cancellative) barycentric algebra $(\mathcal{V},\underline{I}^{\circ})$. Now, consider its subalgebras, i.e. (cancellative) barycentric algebras. Obviously, they are the convex subsets of the real vector space $\mathcal{V}$.
\end{example}

As algebras, \emph{convex polytopes} are defined as finitely generated convex sets. The minimal set of generators of a polytope is the set of its vertices (i.e., its extreme points). In geometric terminology, the convex set generated by a set $V$ is its \emph{convex hull}. If a $k$-dimensional polytope has $n$ vertices, then $n$ is at least $k+1$.

One can treat a convex polytope $\Pi$ in $\mathbb{R}^k$ as a cancellative barycentric algebra $(\Pi,\underline{I}^{\circ})$, a subalgebra of $(\mathbb{R}^k,\underline{I}^{\circ})\in\mathbf{C}$ constructed as described in Example \ref{Ex:VectSpace}. The set of vertices $V$ of $\Pi$ is the generating set of the algebra $(\Pi,\underline{I}^{\circ})$, i.e. $(\Pi,\underline{I}^{\circ})$ is the smallest cancellative barycentric algebra which contains $V$. In other words, any element $\mathbf{p}\in\Pi$ can be obtained by applying finitely many times finitely many operations from $\underline{I}^{\circ}$ to finitely many elements in $V$. Derived operations in $\Pi$ are then just convex combinations of some vertices, and thus barycentric homomorphisms preserve them. Throughout the further parts of the paper it will be assumed that $V$ is the counter-clockwise ordered set of elements $\mathbf v_1 <\dots <\mathbf v_n$.

\begin{example}
Note also that $\mathcal{I}=(I,\underline{I}^{\circ})$ is a convex set, a subalgebra of $(\mathbb{R},\underline{I}^{\circ})$. It provides a host of other examples, starting from the direct product (power) of $n$ copies of $\mathcal{I}$ for any positive integer $n$. Likewise, any space of functions with codomain $\mathcal{I}^n$, for $n\in\mathbb{N}^+$, is a convex set under the pointwise operations \eqref{E:pontwise}.
\end{example}

\subsection{Convex combinations}\label{Ss:Coco}
Let $\Pi$ be a convex set in $\mathbb{R}^k$ treated geometrically. Consider a point $\mathbf{a}\in \Pi\smallsetminus V$ presented as a convex combination $\sum_{i=1}^{r} \alpha_i\mathbf{v}'_i$ with coefficients in $I^\circ$. Here, $\{\mathbf{v}'_1<\ldots<\mathbf{v}'_r\}$ is a potentially proper subset of $V$, with the order inherited from $V$. Then $\mathbf{a}$ can be presented as a composition of the following basic operations in $\underline{I}^{\circ}$: $\underline{p_{r-1}}\left(\underline{p_{r-2}}(\ldots (\underline{p_{1}}(\mathbf{v}'_1,\mathbf{v}'_2),\ldots ,\mathbf{v}'_{r-1}),\mathbf{v}'_r \right)$ with $p_i={\alpha_{i+1}}\big/{\sum_{k=1}^{i+1}\alpha_{k}}$ for $1\leq i\leq r-1$. Now, suppose that we start from a convex set $(\Pi,\underline{I}^{\circ})$, a subalgebra of $(\mathbb{R}^k,\underline{I}^{\circ})$, and
$$
\mathbf{a}=\underline{p_{r-1}}\left(\underline{p_{r-2}}(\ldots (\underline{p_{1}}(\mathbf{v}'_1,\mathbf{v}'_2),\ldots ,\mathbf{v}'_{r-1}),\mathbf{v}'_r \right).
$$
Then $\mathbf{a}$ can be presented as a convex combination $\sum_{i=1}^{r} \alpha_i\mathbf{v}'_i$ with coefficients $\alpha_1=\prod_{k=1}^{r-1}\overline{p_{k}}$, $\alpha_i=p_{i-1}\cdot \prod_{k=i}^{r-1}\overline{p_{k}}$ for $2\leq i\leq r-1$, and $\alpha_r=p_{r-1}$.

\section{Polytope coordinate systems}\label{S:pcs}

\subsection{Barycentric coordinate systems on convex polytopes}\label{Ss:gbc}

Let $\Pi\subseteq\mathbb{R}^k$ be a convex polytope presented as the convex hull of the counter-clockwise ordered sequence $\mathbf v_1,\dots,\mathbf v_n$ of extreme points located around its boundary.

\begin{definition}\cite{FloaterGBCA, WarSchHirDes}
A \emph{barycentric coordinate system} with respect to $\Pi$ is a set $\{b_i\colon\Pi\rightarrow I\mid i=1,\ldots ,n\}$ of functions such that for each point $\mathbf v$ of $\Pi$ the following conditions are satisfied:
\begin{enumerate}
\item[$(\mathrm a)$] \emph{partition of unity:} $\sum_{i=1}^{n}b_i(\mathbf v)=1$;
\item[$(\mathrm b)$] \emph{linear precision:} $\sum_{i=1}^{n}b_i(\mathbf v)\mathbf v_i=\mathbf v$.
\end{enumerate}
\end{definition}

This exactly means that each point of a polygon $\Pi$ can be expressed as a convex combination of the vertices with coefficients $b_i(\mathbf v)$. The idea goes back to A.F.~M\"{o}bius \cite{Moebius} who introduced barycentric coordinates for triangles. They have been generalized in several ways to arbitrary polygons, polyhedra, higher dimensional polytopes, and curves, due to their role in approximation theory and numerical analysis \cite{HS17}.

\begin{remark}
\begin{itemize}
\item[$(\mathrm i)$] Some authors add the \emph{Lagrange property}:
\begin{equation}\label{E:LP}
b_i(\mathbf v_j)=\delta_{ij},
\end{equation}
to the definition of a barycentric coordinate system, where $\delta_{ij}$ is the Kronecker delta. But since here we are considering only convex polytopes, where the $\mathbf v_i$ for $i=1,\ldots n$ are the extreme points of $\Pi$ or, in algebraic language, the \emph{irredundant generators} of the algebra $\Pi$, the condition \eqref{E:LP} is satisfied automatically.
\item[$(\mathrm{ii})$] Another common requirement imposed on the functions $b_i$ for $i=1,\ldots n$ is that of \emph{continuity}. For example, the results in \cite{G} are formulated for continuous functions. We do not make such an assumption here. However, continuity is preserved by convex combinations, so continuous functions form \emph{subalgebras} of the algebras considered in subsequent sections.
\end{itemize}
\end{remark}

\subsection{The convex set of coordinate systems}

\begin{example}\label{Ex:fspace}
Let $\Pi\subseteq \mathbb{R}^k$ be a polytope with an ordered vertex set $V=\set{\mathbf v_1<\dots<\mathbf v_n}$. Consider the real interval $\mathcal{I}=(I,\underline{I}^{\circ})$ as a convex set with the weighted mean operations (see Example~\ref{Ex:VectSpace}). The space $\mathbf{Set}(\Pi,\mathcal{I})$ is a convex set (cancellative barycentric algebra) $(\mathbf{Set}(\Pi,\mathcal{I}),\underline{I}^{\circ})$. We can apply the construction above again, now to a vertex set $V$ of $\Pi$ and $(\mathbf{Set}(\Pi,\mathcal{I}),\underline{I}^{\circ})\in \mathbf{C}$ and obtain a new convex set $(\mathbf{Set}(V,\mathbf{Set}(\Pi,\mathcal{I})),\underline{I}^{\circ})$. The latter is isomorphic to $(\mathbf{Set}(\Pi\times V,\mathcal{I}),\underline{I}^{\circ})\in\mathbf{C}$ by the Currying isomorphism 
$$
\iota\colon \mathbf{Set}(\Pi\times V,\mathcal{I})\rightarrow \mathbf{Set}(V,\mathbf{Set}(\Pi,\mathcal{I}))
$$
taking a function
$$
h\colon \Pi\times V\rightarrow I;\;(\mathbf x,\mathbf v)\mapsto h(\mathbf x,\mathbf v);
$$
to a family of functions $(h_{\mathbf v})_{\mathbf{v}\in V}$ with $h_\mathbf{v}\colon \Pi\rightarrow I;\;\mathbf{x}\mapsto h(\mathbf{x},\mathbf{v})$. Hence, we can identify the barycentric algebras $\mathbf{Set}(V,\mathbf{Set}(\Pi,\mathcal{I}))$ and $\mathbf{Set}(\Pi\times V,\mathcal{I})$.
\end{example}

\begin{definition}\cite[Definition 6.3]{RSZ1}
Let $\Pi$ be a polytope with a vertex set $V=\set{\mathbf v_1,\dots,\mathbf v_n}$.  A \emph{coordinate system} for $\Pi$ is a map
\begin{equation}
\lambda\colon V\to \mathbf{Set}(\Pi,\mathcal{I});\mathbf v\mapsto\lambda_{\mathbf v}
\end{equation}
such that $\mathbf{a}=\sum_{\mathbf{v}\in V}\lambda_{\mathbf{v}}(\mathbf{a})\mathbf{v}_i$, i.e., the linear precision property holds, for all $\mathbf{a}\in\Pi$.
\end{definition}

\begin{remark}
In our algebraic setting, the partition of unity property, namely $\sum_{\mathbf{v}\in V}\lambda_{\mathbf{v}}(\mathbf{a})=1$, follows from the linear precision property \cite[Lemma 2.23]{RSZ1}. Indeed, let $\mathbf{a}\in\Pi$ and $\lambda_{\mathbf{v}}$, $\mathbf{v}\in V$. Take the constant function $\mathds 1_\Pi\colon \Pi\to \{1\}$. Note $\mathds 1_\Pi\in\mathbf{B}(\Pi,\mathcal{I})$, since $\{1\}\leq \mathcal{I}$. Hence,
\begin{align*}
\mathds 1_\Pi(\mathbf{a})=\mathds 1_\Pi\left(\sum_{\mathbf{v}\in V}\lambda_{\mathbf{v}}(\mathbf{a})\mathbf{v}_i\right)=\sum_{\mathbf{v}\in V}\lambda_{\mathbf{v}}(\mathbf{a})\mathds 1_\Pi(\mathbf{v}_i),
\end{align*}
and the partition of unity property follows.
\end{remark}

The set $K_\Pi$ of cooordinate systems on a polygon $\Pi$ is a subalgebra of $(\mathbf{Set}(\Pi\times V,\mathcal{I}),\underline{I}^{\circ})$ under pointwise barycentric operations \cite[Theorem 6.6]{RSZ1}. It forms a convex set $(K_\Pi,\underline{I}^{\circ})$. This raises further questions on the structure of the set. In particular,

\begin{question}
What are the extreme points (irredundant generators) of the convex set $(K_\Pi,\underline{I}^{\circ})$?
\end{question}

\begin{example}\label{L:Kalman}
Let $\Pi$ be a convex polytope in $\mathbb{R}^k$ understood as a convex set $(\Pi,\underline{I}^\circ)$ finitely generated by a set $V=\{\mathbf v_1,\ldots \mathbf v_n\}$ of vertices. Then each element $\mathbf a\in \Pi$ can be represented as a convex combination of vertices. The combination is not unique if $\Pi$ is not a simplex, but choose any such combination $\mathbf a=\sum_{i=1}^na_i\mathbf v_i$. Now define functions $\lambda_i\in \mathbf{Set}(\Pi,\mathcal{I})$ by $\lambda_i(\mathbf a)=a_i$. Then $\lambda=(\lambda_1,\ldots \lambda_n)$ is a cooordinate system on $\Pi$.
\end{example}

\begin{remark}
Continuous versions of the fact shown in Example \ref{L:Kalman} that each element in $(\Pi,\underline{I}^\circ)$ can be represented by a (continuous) coordinate system in $K_\Pi$ can be found in \cite[Theorem~2]{K}, \cite[Theorem~9.8.1]{Modes}, \cite[Lemma 2.2]{G}.
\end{remark}

\section{The tautological map}\label{S:tm}

In \cite{G} links between barycentric coordinate systems and certain clas\-ses of partitions of unity on a given convex polytope $\Pi$ in $\mathbb{R}^k$ are discussed. Motivated by applications, \cite{G} focuses on continuous functions, but as discussed above, we will not make this assumption.
The aim of this section is to interpret the results of \cite{G} in the language of barycentric algebras and show how they can be obtained within this theory.

Let $\Pi$ be a convex polytope in $\mathbb{R}^k$ with $n$ vertices, considered as a barycentric algebra. 
A function $f\in\mathbf{Set}(\Pi,\mathcal{I}^n)$ will be now specified by the element $(f_1,\ldots ,f_n)$ of the pointwise barycentric algebra $\mathbf{Set}(\Pi,\mathcal{I})^n$.

By $\mathbf{Set}^{\mathds 1}(\Pi,\mathcal{I}^n)$ we denote the set of all functions from the set $\Pi$ to $\mathcal{I}^n\in\mathbf{B}$ which have the partition of unity property, i.e.
\begin{equation*}
f\in\mathbf{Set}^{\mathds 1}(\Pi,\mathcal{I}^n) \mathrel{\Leftrightarrow} f\in\mathbf{Set}(\Pi,\mathcal{I}^n) \text{ and } \sum_{i=1}^n f_i(a)=1 \text{ for all } a\in\Pi.
\end{equation*}

\begin{lemma}
$\mathbf{Set}^{\mathds 1}(\Pi,\mathcal{I}^n)$ is a subalgebra of the barycentric algebra $\mathbf{Set}(\Pi,\mathcal{I}^n)$.
\end{lemma}

\begin{proof}
For $f,g\in\mathbf{Set}^{\mathds 1}(\Pi,\mathcal{I}^n)$ and $q\in I^\circ$, one has
\begin{align*}
\underline{q}(f,g)&=\underline{q}\left((f_1,f_2,\ldots ,f_n),(g_1,g_2,\ldots g_n)\right)\\
&=\left(\underline{q}(f_1,g_1),\underline{q}(f_2,g_2),\ldots \underline{q}(f_n,g_n)\right).
\end{align*}
Then
\begin{align*}
\sum_{i=1}^{n}\underline{q}(f_i,g_i)=\underline{q}\left(\sum_{i=1}^{n}f_i,\sum_{i=1}^{n}g_i\right)=\underline{q}(\mathds 1_\Pi,\mathds 1_\Pi)=\mathds 1_\Pi.
\end{align*}
\end{proof}

In a similar way, one can show that $\mathbf{Set}_{LP}^{\mathds 1}(\Pi,\mathcal{I}^n)$, consisting of all partitions of unity with the Lagrange property, also forms a convex set. Hence, we obtain the following sequence of subalgebras in $\mathbf{C}$:
\begin{equation}\label{E:sequence}
\mathbf{Set}_{LP}^{\mathds 1}(\Pi,\mathcal{I}^n)\leq \mathbf{Set}^{\mathds 1}(\Pi,\mathcal{I}^n)\leq \mathbf{Set}(\Pi,\mathcal{I}^n).
\end{equation}

\begin{definition}\cite[(2.6)]{G}
The mapping
\begin{align*}
T\colon \mathbf{Set}^{\mathds 1}(\Pi,\mathcal{I}^n)&\to \mathbf{Set}(\Pi,\mathbb{R}^k);\\
f&\mapsto \left(T_f\colon \Pi\to \mathbb{R}^k;\mathbf{a}\mapsto\sum_{i=1}^nf_i(\mathbf{a})\mathbf{v}_i\right)
\end{align*}
is called the \emph{tautological map}.
\end{definition}
\begin{lemma}
The tautological map $T$ is a barycentric homomorhism.
\end{lemma}
\begin{proof}
For $f,g\in \mathbf{Set}^{\mathds 1}(\Pi,\mathcal{I}^n)$, $q\in I^\circ$ and $\mathbf{a}\in \Pi$, one has
\begin{align*}
T\left(\underline{q}(f,g)\right)(\mathbf{a})&=\sum_{i=1}^n\underline{q}(f_i,g_i)(\mathbf{a})\mathbf{v}_i=\sum_{i=1}^n\underline{q}(f_i(\mathbf{a})\mathbf{v}_i,g_i(\mathbf{a})\mathbf{v}_i)\\
&=\underline{q}(\sum_{i=1}^nf_i(\mathbf{a})\mathbf{v}_i,\sum_{i=1}^ng_i(\mathbf{a})\mathbf{v}_i)=\underline{q}\left(T(f),T(g)\right)(\mathbf{a}).
\end{align*}
\end{proof}
\begin{corollary}\cite[Theorem 6.6]{RSZ1}\label{C:KPi}
The set $K_\Pi$ of cooordinate systems on a polytope $\Pi$ with vertex set $V$ forms a convex subset of $\mathbf{Set}^{\mathds 1}(\Pi,\mathcal{I}^n)$ under pointwise barycentric operations.
\end{corollary}
\begin{proof}
Due to idempotence, the singleton $\{1_\Pi\}$ is a subalgebra in the convex set $\mathbf{Set}(\Pi,\mathbb{R}^k)$. Its homomorphic preimage  $T^{-1}(\{1_\Pi\})$, as a subalgebra in the convex set $\mathbf{Set}^{\mathds 1}(\Pi,\mathcal{I}^n)$, is itself a convex set. Obviously, $T^{-1}(\{1_\Pi\})=K_\Pi$ (see also \cite[Proposition 2.1]{G}).
\end{proof}
In fact, direct calculations show that $K_\Pi\leq\mathbf{Set}_{LP}^{\mathds 1}(\Pi,\mathcal{I}^n)$ in the sequence \eqref{E:sequence}. The following lemma describes when the algebras coincide.

\begin{lemma}\cite[Theorem 2.3]{G}
Suppose $f\in T^{-1}\left(\mathbf{B}(\Pi,\mathbb{R}^k)\right)$. Then $f$ has the Lagrange property if and only if $f$ is a barycentric coordinate system.
\end{lemma}

\begin{proof}
By the Lagrange property of $f$ we obtain $T_f(\mathbf{v}_i)=\mathbf{v}_i$ for $i=1,\ldots n$. Recall that any element of a convex polytope can be represented as a convex combination of the vertices. Let $\mathbf{a}\in\Pi$ and let $\lambda_i\in I$ for $i=1,\ldots ,n$ be such that $\mathbf{a}=\sum_{i=1}^n\lambda_i\mathbf{v}_i$. Since $T_f$ is a barycentric homomorphism, we obtain $T_f(\mathbf{a})=T_f(\sum_{i=1}^n\lambda_i\mathbf{v}_i)=\sum_{i=1}^n\lambda_iT_f(\mathbf{v}_i)=\sum_{i=1}^n\lambda_i\mathbf{v}_i$. Hence $T_f=1_\Pi$ and $f\in K_\Pi$.
\end{proof}

In \cite[Section 3]{G} further properties of the tautological map $T$ have been investigated. We describe some of them in the language of bary\-centric algebra theory.
\begin{corollary}
Let $\Pi\subset \mathbb{R}^k$ be a convex polytope on $n$ vertices.
\begin{enumerate}
\item[$(\mathrm a)$] \cite[Theorem 3.2]{G}: $T\left(\mathbf{Set}^{\mathds 1}(\Pi,\mathcal{I}^n)\right)=\mathbf{Set}(\Pi,\Pi)$;
\item[$(\mathrm b)$] \cite[Theorem 3.4 (i)-(ii)]{G}: $T\left(\mathbf{Set}_{LP}^{\mathds 1}(\Pi,\mathcal{I}^n)\right)=\{h\in\mathbf{Set}(\Pi,\Pi)\colon\\ h(\mathbf{v}_i)=\mathbf{v}_i,\, i=1,\ldots ,n\}$;
\item[$(\mathrm c)$] \cite[Theorem 3.4 (iii)-(iv)]{G}: $T\left(\mathbf{B}_{LP}^{\mathds 1}(\Pi,\mathcal{I}^n)\right)=\{h\in\mathbf{B}(\Pi,\Pi)\colon\\ h(\mathbf{v}_i)=\mathbf{v}_i,\, i=1,\ldots ,n\}=T(K_\Pi)=\{1_\Pi\}$;
\item[$(\mathrm d)$] \cite[Proposition 3.5]{G}: $\mathbf{B}(\Pi,\Pi)=\{h\in\mathbf{B}(\Pi,\mathbb{R}^k)\colon h(V)\subset\Pi\}$.
\end{enumerate}
\end{corollary}
\begin{proof}
$(\mathrm a)$ ($\Rightarrow$) Let $f=T_p$ for some $p\in\mathbf{Set}^{\mathds 1}(\Pi,\mathcal{I}^n)$. It follows that $f(\mathbf a)=\sum_{i=1}^np_i(\mathbf a)\mathbf v_i$ for any $\mathbf a\in\Pi$ and, as a convex combination of vertices of $\Pi$, $f(\mathbf a)\in\Pi$.\vskip2pt
($\Leftarrow$) Now, let $f\in\mathbf{Set}(\Pi,\Pi)$. Let $\lambda\in K_\Pi$ be as described in Example \ref{L:Kalman}. For $i=1,\ldots ,n$ take $p_i=\lambda_i\circ f$. Then $p=(p_1,\ldots, p_n)$ has the partition of unity property, since $\lambda$ has it. Moreover, $\sum_{i=1}^np_i(\mathbf a)\mathbf v_i= \sum_{i=1}^n \lambda_i(f(\mathbf a))\mathbf v_i=f(\mathbf a)$ for any $\mathbf a\in\Pi$. Hence, $f=T_p$.\vskip2pt
$(\mathrm b)$ is proved the same way as $(\mathrm a)$. \vskip2pt
Items $(\mathrm c)$ and $(\mathrm d)$ follow immediately by Proposition \ref{P:hs}.
\end{proof}

\end{document}